\documentclass[10pt]{amsart}
\usepackage{amsmath}
\usepackage{amssymb}
\usepackage{xypic}
\usepackage{yhmath}
\usepackage{geometry}
\def\beqnn{\begin{eqnarray*}}\def\eeqnn{\end{eqnarray*}}

\newtheorem{theorem}{Theorem}[section]
\newtheorem{proposition}[theorem]{Proposition}
\newtheorem{corollary}[theorem]{Corollary}
\newtheorem{lemma}[theorem]{Lemma}
\newtheorem{problem}[theorem]{Problem}

\theoremstyle{definition}

\newtheorem{remark}[theorem]{Remark}

\newtheorem{claim}[theorem]{Claim}

\theoremstyle{question}

\numberwithin{equation}{section}

\begin{document}

\title[Generalized Hilbert matrix operators]{Generalized Hilbert matrix operators acting on\\ weighted sequence spaces}


\author{Jianjun Jin}
\address{School of Mathematics, Hefei University of Technology, Xuancheng Campus, Xuancheng 242000, P.R.China}
\email{ jin@hfut.edu.cn, jinjjhb@163.com}


\thanks{The author was supported by National Natural Science Foundation of China(Grant Nos. 11501157). }






\begin{abstract}
In this paper we introduce and study a new kind of generalized Hilbert matrix operators, induced by a positive finite Borel measure on $(0,1)$, acting on weighted sequence spaces.  We establish a sufficient and necessary condition for the boundedness of these operators. These results extend some related ones obtained recently in [Bull. Lond. Math. Soc., 55 (2023), no. 6, 2598–2610].
\end{abstract}

\maketitle

\section{{\bf Introductions and main results}}

For $1<p<\infty$, we denote the H\"older conjugate of $p$ by $q$, i.e., $\frac{1}{p}+\frac{1}{q}=1$.  Let $l^p$ be the space of sequences of complex numbers, i.e.,
\begin{equation*}l^{p}:=\{a=\{a_n\}_{n=0}^{\infty}: \|a\|_{p}=(\sum_{n=0}^{\infty} |a_{n}|^p)^{\frac{1}{p}}<\infty \}.\end{equation*}
For $a=\{a_n\}_{n=0}^{\infty}\in l^{p}$, $b=\{b_n\}_{n=0}^{\infty}\in l^{q}$,  the famous Hilbert's inequality is stated as 
$$\Big|\sum_{m=0}^{\infty}\sum_{n=0}^{\infty}\frac{a_m b_n}{m+n+1}\Big|\leq \pi\csc(\frac{\pi}{p})\|a\|_p\|b\|_q,$$
where the constant $\pi\csc(\frac{\pi}{p})$ is the best possible. Hilbert's inequality has an equivalent form as follows.
\begin{equation}\label{h} \sum_{n=0}^{\infty}\Big|\sum_{m=0}^{\infty}\frac{a_m}{m+n+1}\Big|^p  \leq \Big[\pi\csc(\frac{\pi}{p})\Big]^p\|a\|_p^p,\end{equation}
where the constant $\Big[\pi\csc(\frac{\pi}{p})\Big]^p$ is the best possible, see \cite[Theorem 6]{HLP-0} or \cite[Theorem 323]{HLP}.

The inequality (\ref{h}) can be restated in the language of operator theory. For the Hilbert kernel $\frac{1}{m+n+1}$, we define the Hilbert matrix operator $\mathbf{H}$ as
$${\bf{H}}(a)(n):=\sum_{m=0}^{\infty}\frac{a_m}{m+n+1}, \,a=\{a_m\}_{m=0}^{\infty}, \,n\in \mathbb{N}_0=\mathbb{N} \cup \{0\}.$$
Then  (\ref{h}) can be restated as
\begin{theorem}\label{m-0}
Let $1<p<\infty$. Then $\mathbf{H}$ is bounded on $l^p$ and the norm of $\mathbf{H}$ is $\pi\csc(\frac{\pi}{p})$. \end{theorem}

The Hilbert matrix operator is important in analysis and there have been many results about this operator and its analogues and generalizations. The history and classical results of this operator can be found in the famous monograph \cite{HLP}. Modern results about generalizations of the Hilbert matrix operator and their analogues can be found in the survey \cite{YR} and Yang's book \cite{Y3}. Some very new results about the Hilbert matrix operator have been obtained in \cite{DGP}. In \cite{D}, \cite{DS}, Diamantopoulos and Siskakis initiated the study of the Hilbert matrix operator acting on the analytic function spaces. Later, there have been a good number of researchers studying the Hilbert matrix operator and its generalizations on various spaces of analytic functions, see, for example, \cite{BK}, \cite{Dai}, \cite{D-1}, \cite{DJV}, \cite{GGPS}, \cite{GP}, \cite{GM}, \cite{GM-2}, \cite{JTF}, \cite{Ka}, \cite{LMN}, \cite{LMW}, \cite{LS},\cite{N}, \cite{PR}, \cite{TZ}, \cite{YZ1}. For more results about these topics, see the recent survey \cite{bds}.   

Note that the Hilbert kernel can be written as 
$$\frac{1}{m+n+1}=\binom{n+m}{m}\int_{0}^{1}t^m (1-t)^ndt.$$
Here, for $\gamma \in \mathbb{R}$, $m\in \mathbb{N}_0$, the combinatorial number $\binom{\gamma}{m}$ is defined as 
$$\binom{\gamma}{m}:=\frac{\gamma(\gamma-1)\cdots (\gamma-m+1)}{m!}.$$
In particular,  $\binom{\gamma}{0}=1$ for any $\gamma \in \mathbb{R}$. 

Very recently, inspired by the classical work of Hardy \cite{har}, the following generalized Hilbert matrix operator was introduced and studied by Athanasiou in \cite{Ath}. Let $\mu$ be a positive finite Borel measure on $(0,1)$, the
operator $H_{\mu}$ is defined as
$$H_{\mu}(a)(n):=\sum\limits_{m=0}^{\infty}\int_{(0,1)}\binom{n+m}{m}t^m (1-t)^{n}a_md\mu(t), a=\{a_m\}_{m=0}^{\infty},\, n \in \mathbb{N}_0.$$

It has been proved in \cite{Ath} that
\begin{theorem}\label{ath}
Let $\mu$ be a positive finite Borel measure on $(0,1)$. 

{\bf(1)} Let $1\leq p<+\infty$.  Then $H_{\mu}$ is bounded on $l^{p}$ if and only if 
$$\mathcal{C}_{\mu}(p)=\int_{(0,1)} (1-t)^{\frac{1}{p}-1} t^{-\frac{1}{p}}d\mu(t)<\infty.$$
Moreover, when $H_{\mu}:l^{p}\rightarrow l^{p}$ is bounded, the norm of $H_{\mu}$ is $\mathcal{C}_{\mu}(p).$

{\bf(2)} $H_{\mu}$ is bounded on $l^{\infty}$ if and only if 
$$\mathcal{C}_{\mu}(\infty)=\int_{(0,1)} (1-t)^{-1}d\mu(t)<\infty.$$
Moreover, when $H_{\mu}$ is bounded on $l^{\infty}$, the norm of $H_{\mu}$ is $\mathcal{C}_{\mu}(\infty).$
 
\end{theorem}

The main purpose of this note is to establish an extension of above results obtained in \cite{Ath}.  To state our results, we first introduce some notations. Let $w(n)\geq 0$ for $n\in \mathbb{N}_0$. For $1\leq p<\infty$, we denote by $l_{w}^{p}$ the weighted Lesbegue space of infinite sequences, i.e.,
\begin{equation*}l_{w}^{p}:=\{a=\{a_n\}_{n=0}^{\infty}: \|a\|_{p,w}=(\sum_{n=0}^{\infty} w(n)|a_{n}|^p)^{\frac{1}{p}}<\infty \}.\end{equation*}
We define the class of infinite sequences $l_{w}^{\infty}$ as 
\begin{equation*}l_{w}^{\infty}:=\{a=\{a_n\}_{n=0}^{\infty}: \|a\|_{\infty, w}=\sup_{n\in \mathbb{N}_0}w(n)|a_{n}|<\infty \}.\end{equation*}
When $w(n)\equiv1$, we will write $l^{p}$ instead of $l_{w}^{p}$, where $1\leq p \leq \infty$. 

Let $\alpha, \beta>-1$, $\beta-\alpha>-1$ and let $\mu$ be a positive finite Borel measure on $(0,1)$. For a sequence $a=\{a_m\}_{m=0}^{\infty}$, we define
$$H_{\mu}^{\alpha, \beta}(a)(n):=\sum\limits_{m=0}^{\infty}\int_{(0,1)}\frac{\Gamma(n+m+\beta+1)}{\Gamma(m+\alpha+1)\Gamma(n+\beta-\alpha+1)}t^m (1-t)^{n}a_md\mu(t), \,  n \in \mathbb{N}_0.$$
Here, $\Gamma$ is the usual Gamma function, see \cite{AAR}. When $\alpha=0$, $H_{\mu}^{\alpha, \beta}$ reduces to the following operator
$$H_{\mu}^{\beta}(a)(n):=\sum\limits_{m=0}^{\infty}\int_{(0,1)}\binom{n+m+\beta}{m}t^m (1-t)^{n}a_md\mu(t),\,  n \in \mathbb{N}_0.$$

For two real numbers $s, \gamma$ with $\gamma>-1$ and $\gamma+s>-1$, we define  
$$(\gamma+1)_s:=\frac{\Gamma(\gamma+s+1)}{\Gamma(\gamma+1)}.$$

We shall prove that
\begin{theorem}\label{main}
Let $\mu$ be a positive finite Borel measure on $(0,1)$. Let $1\leq p<\infty$, and $\alpha, \beta>-1$, $\beta-\alpha>-1$. Then $H_{\mu}^{\alpha,\beta}$ is bounded from $l_{w_1}^{p}$ into $l_{w_2}^{p}$ if and only if 
$$\mathcal{C}_{\mu}(\beta, p):=\int_{(0,1)} (1-t)^{-(1-\frac{1}{p})(\beta+1)} t^{-\frac{1}{p}(\beta+1)}d\mu(t)<\infty.$$
Moreover, when $H_{\mu}^{\alpha, \beta}:l_{w_1}^{p}\rightarrow l_{w_2}^{p}$ is bounded, the norm $\|H_{\mu}^{\alpha,\beta}\|$ of $H_{\mu}^{\alpha, \beta}$ is $\mathcal{C}_{\mu}(\beta, p).$
Here $$w_1(m)=(m+1)_{\alpha}^{-p} (m+1)_{\beta},$$
 $$w_2(m)=(m+\beta-\alpha+1)_{\alpha}^{-p} (m+1)_{\beta}, \,\, m\in \mathbb{N}_0. $$
\end{theorem}

\begin{theorem}\label{main-2}
Let $\mu$ be a positive finite Borel measure on $(0,1)$. Let $\alpha, \beta>-1$, $\beta-\alpha>-1$. Then $H_{\mu}^{\alpha,\beta}$ is bounded from $l_{\overline{w}_1}^{\infty}$ into $l_{\overline{w}_2}^{\infty}$ if and only if 
$$\mathcal{C}_{\mu}(\beta, \infty):=\int_{(0,1)} (1-t)^{-\beta-1}d\mu(t)<\infty.$$
Moreover, when $H_{\mu}^{\alpha, \beta}: l_{\overline{w}_1}^{\infty}\rightarrow l_{\overline{w}_2}^{\infty}$ is bounded, the norm $\|H_{\mu}^{\alpha,\beta}\|$ of $H_{\mu}^{\alpha, \beta}$ is $\mathcal{C}_{\mu}(\beta,\infty).$
Here $$\overline{w}_1(m)=(m+1)_{\alpha}^{-1}, \, \overline{w}_2(m)=(m+\beta-\alpha+1)_{\alpha}^{-1}, \,\, m\in \mathbb{N}_0.$$
\end{theorem}

When $\alpha=0$, from Theorem \ref{main} and \ref{main-2},  we have
\begin{corollary}\label{cor}
Let $\mu$ be a positive finite Borel measure on $(0,1)$. Let $\beta>-1$.

{\bf(1)} Let $1\leq p<\infty$. Then $H_{\mu}^{\beta}$ is bounded on $l_{w}^{p}$ if and only if 
$$\mathcal{C}_{\mu}(\beta, p)=\int_{(0,1)} (1-t)^{-(1-\frac{1}{p})(\beta+1)} t^{-\frac{1}{p}(\beta+1)}d\mu(t)<\infty.$$
Moreover, when $H_{\mu}^{\beta}$ is bounded on $l_{w}^{p}$, the norm of $H_{\mu}^{\beta}$ is $\mathcal{C}_{\mu}({\beta,p}).$ Here, $$w(m)=(m+1)_{\beta},\,\, m\in \mathbb{N}_0.$$

{\bf(2)} $H_{\mu}^{\beta}$ is bounded on $l^{\infty}$ if and only if 
$$\mathcal{C}_{\mu}(\beta, \infty)=\int_{(0,1)} (1-t)^{-\beta-1}d\mu(t)<\infty.$$
Moreover, when $H_{\mu}^{\beta}$ is bounded on $l^{\infty}$, the norm of $H_{\mu}^{\beta}$ is $\mathcal{C}_{\mu}({\beta,\infty}).$

\end{corollary}
\begin{remark}
Theorem \ref{ath} will follow if we take $\beta=0$ in Corollary \ref{cor}.  
\end{remark}

The paper is organized as follows. Some lemmas will be given in the next section.  We will first prove Theorem \ref{main} in Section 3. The proof of Theorem \ref{main-2} will be given in Section 4.  Final remarks will be presented in Section 5.

\section{{\bf Some lemmas}}
To prove the main results, in this section, we will list some known lemmas and establish some new ones. 
For the sake of simplicity, in the rest of the paper, we will write 
$$k_{\alpha,\beta}(m,n)=\frac{\Gamma(n+m+\beta+1)}{\Gamma(m+\alpha+1)\Gamma(n+\beta-\alpha+1)}.$$
\begin{lemma}\label{l-1}
Let $m, n \in \mathbb{N}_0$ and let $\alpha, \beta>-1$, $\beta-\alpha>-1$. Then we have
$$k_{\alpha,\beta}(m,n)=(-1)^m \binom{-n-\beta-1}{m}(m+1)_{\alpha}^{-1}(n+\beta-\alpha+1)_{\alpha},$$
and 
$$k_{\alpha,\beta}(m,n)=(-1)^n \binom{-m-\beta-1}{n}(n+1)_{\beta-\alpha}^{-1}(m+\alpha+1)_{\beta-\alpha}.$$
\end{lemma}
\begin{proof}
Note that \begin{eqnarray}\Gamma(n+m+\beta+1)&=&\Gamma(n+\beta+1)\prod_{i=1}^{m}(n+\beta+i)\nonumber \\
&=& \Gamma(n+\beta+1)(-1)^m\prod_{i=1}^{m}(-n-\beta-i). \nonumber\end{eqnarray}
Then
\begin{eqnarray}\lefteqn{k_{\alpha,\beta}(m,n)=\frac{\Gamma(n+m+\beta+1)}{\Gamma(n+\beta-\alpha+1)}\cdot\frac{\Gamma(m+1)}{\Gamma(m+\alpha+1)}\cdot\frac{1}{\Gamma(m+1)}}\nonumber \\
&&= (-1)^m \frac{\prod_{i=1}^{m}(-n-\beta-i)}{\Gamma(m+1)}\cdot\frac{\Gamma(m+1)}{ \Gamma(m+\alpha+1)}\cdot\frac{\Gamma(n+\beta+1)}{\Gamma(n+\beta-\alpha+1)} \nonumber \\
&&=(-1)^m \binom{-n-\beta-1}{m}(m+1)_{\alpha}^{-1}(n+\beta-\alpha+1)_{\alpha}. \nonumber
\end{eqnarray}
Similarly, we have \begin{eqnarray}\Gamma(n+m+\beta+1)&=&\Gamma(m+\beta+1)\prod_{i=1}^{n}(m+\beta+i)\nonumber \\
&=& \Gamma(m+\beta+1)(-1)^n \prod_{i=1}^{n}(-m-\beta-i). \nonumber  \end{eqnarray}
It follows that
\begin{eqnarray}\lefteqn{k_{\alpha,\beta}(m,n)=\frac{\Gamma(n+m+\beta+1)}{\Gamma(m+\alpha+1)}\cdot\frac{\Gamma(n+1)}{ \Gamma(n+\beta-\alpha+1)}\cdot\frac{1}{\Gamma(n+1)}} \nonumber \\
&&= (-1)^n \frac{\prod_{i=1}^{n}(-m-\beta-i)}{\Gamma(n+1)}\cdot\frac{\Gamma(n+1)}{ \Gamma(n+\beta-\alpha+1)}\cdot\frac{\Gamma(m+\beta+1)}{\Gamma(m+\alpha+1)} \nonumber \\
&&=(-1)^n \binom{-m-\beta-1}{n}(n+1)_{\beta-\alpha}^{-1}(m+\alpha+1)_{\beta-\alpha}. \nonumber
\end{eqnarray}
This proves the lemma.
\end{proof}

\begin{lemma}\label{l-2} Let $m, n \in \mathbb{N}_0$ and let $\alpha, \beta>-1$, $\beta-\alpha>-1$. Then we have
\begin{equation}
\sum_{m=0}^{\infty}k_{\alpha, \beta}(m,n)t^m(m+1)_{\alpha}=(1-t)^{-n-\beta-1}(n+\beta-\alpha+1)_{\alpha}, \nonumber 
\end{equation} 
and
\begin{equation}
\sum_{n=0}^{\infty}k_{\alpha, \beta}(m,n)(1-t)^n(n+1)_{\beta-\alpha}=t^{-m-\beta-1}(m+\alpha+1)_{\beta-\alpha}. \nonumber 
\end{equation} \end{lemma}

\begin{proof}
By Lemma \ref{l-1}, we have
\begin{eqnarray}
\lefteqn{\sum_{m=0}^{\infty}k_{\alpha, \beta}(m,n)t^m(m+1)_{\alpha}}\nonumber \\ 
&&=\sum_{m=0}^{\infty}(-1)^m t^m \binom{-n-\beta-1}{m}(n+\beta-\alpha+1)_{\alpha}\nonumber \\
&&=(1-t)^{-n-\beta-1}(n+\beta-\alpha+1)_{\alpha}\nonumber,
\end{eqnarray}
and 
\begin{eqnarray}
\lefteqn{\sum_{n=0}^{\infty}k_{\alpha, \beta}(m,n)(1-t)^n(n+1)_{\beta-\alpha}}\nonumber \\
&&=\sum_{n=0}^{\infty}(-1)^n (1-t)^n \binom{-m-\beta-1}{n}(m+\alpha+1)_{\beta-\alpha}  \nonumber \\
&&=t^{-m-\beta-1}t^m(m+\alpha+1)_{\beta-\alpha}. \nonumber
\end{eqnarray}
The lemma is proved.
\end{proof}

\begin{lemma}\cite[Lemma 1.4]{Ath}\label{l-5} It holds that $$(1-t)(1-te^{-x})^{-1}\geq e^{-t(1-t)^{-1}x},$$ 
for all $x\geq 0$ and $t\in[0,1)$. \end{lemma}

\begin{lemma}\cite[Lemma 1.5]{Ath}\label{l-4} Let $y,z>0$, then $$y^{-z}=\frac{1}{\Gamma(z)}\int_{0}^{\infty}e^{-yx}x^{z-1}dx.$$ \end{lemma}

\begin{lemma}\cite[Lemma 1]{Con}\label{l-3} Let $n\geq 0$. Then $$(n+1)_{s}\leq \frac{(n+1)^{s+1}}{n+s+1},$$
for $s\in (-1,0),$
and  $$(n+1)_{s}\leq (n+s+1)^{s},$$
for $s\geq 0$.\end{lemma}

\begin{lemma}\label{l-6}
Let $p\geq 1, \beta>-1$. Let $\rho\in (0,1)$ and $\mathbf{N}$ be a natural number. For $\varepsilon\in (0, \frac{1}{p}(\beta+1))$, we define $$a_n(\varepsilon)=(n+\beta+1)^{-1-p\varepsilon},\, n\in \mathbb{N}_0.$$ 
Then we can take a constant $\varepsilon_1=\varepsilon_1(\rho, \mathbf{N})>0$ such that
$$\sum_{n=\mathbf{N}+1}^{\infty}a_n(\varepsilon)\geq (1-\rho)\sum_{n=0}^{\infty}a_n(\varepsilon),$$
for all $\varepsilon\in (0, \varepsilon_1]$. 
\end{lemma}

\begin{proof}
We let $\mathbf{S}=\sum_{n=0}^{\infty}a_n(\varepsilon)$. It is easy to see that
\begin{equation}
\mathbf{S}\geq \int_{0}^{\infty}(x+\beta+1)^{-1-p\varepsilon}dx=\frac{1}{\varepsilon} \frac{1}{p(\beta+1)^{p\varepsilon}}, \nonumber\end{equation}
and \begin{equation}
\sum_{n=0}^{\bf{N}}a_n(\varepsilon)\leq \frac{{\bf{N}}+1}{(\beta+1)^{1+p\varepsilon}}. \nonumber
\end{equation}
This means that we can find two constants $C_i=C_i(p,\beta)>0, i=1,2$, such that
$$\frac{1}{C_1\varepsilon}\leq \mathbf{S},\,\,\,{\text{and}}\,\,\,\, \sum_{n=0}^{\bf{N}}a_n(\varepsilon)\leq C_2{\bf{N}}.$$
It follows that
\begin{equation} 
\Big(\sum_{n={\bf{N}}+1}^{\infty}a_n(\varepsilon)\Big)/\Big(\sum_{n=0}^{\infty}a_n(\varepsilon)\Big)=1-\Big(\sum_{n=0}^{\bf{N}}a_n(\varepsilon)\Big)/\mathbf{S}\geq 1-C_1 C_2{\bf{N}}\varepsilon. \nonumber 
\end{equation}
Hence, if $0<\varepsilon\leq \varepsilon_1:=\frac{\rho}{C_1C_2{\bf{N}}}$, then we have
$$\sum_{n=\mathbf{N}+1}^{\infty}a_n(\varepsilon)\geq (1-\rho)\sum_{n=0}^{\infty}a_n(\varepsilon),$$
for all $\varepsilon\in (0,\varepsilon_1]$. This proves the lemma. 
\end{proof}

\section{{\bf Proof of Theorem \ref{main}}}
First we prove the following
\begin{proposition}\label{i}
If $\mathcal{C}_{\mu}(\beta,p)<\infty$, then $H_{\mu}^{\alpha, \beta}: l_{w_1}^{p}\rightarrow l_{w_2}^p$ is bounded and $\|H_{\mu}^{\alpha, \beta}\|\leq \mathcal{C}_{\mu}(\beta,p)$.  
\end{proposition}
We then prove 
\begin{proposition}\label{ii} 
If $H_{\mu}^{\alpha, \beta}: l_{w_1}^{p}\rightarrow l_{w_2}^p$ is bounded, then $\|H_{\mu}^{\alpha, \beta}\|\geq \mathcal{C}_{\mu}(\beta,p)$.\end{proposition}
It is easy to see that Proposition \ref{i} and \ref{ii} imply Theorem \ref{main}. 

\subsection{Proof of Proposition \ref{i}}
 

For $n\in \mathbb{N}_0$, $a=\{a_m\}_{m=0}^{\infty}\in l_{w_1}^{p}$, let 
$$E_n(t):=\sum_{m=0}^{\infty}k_{\alpha, \beta}(m,n)t^m (1-t)^n |a_m| $$

First, we show that $E_n(t)$ is a well-defined function on $(0,1)$ for each $n\in \mathbb{N}_0$.
When $p=1$, for $a\in l_{w_1}^{1}$, we have
\begin{equation}
E_n(t)=\sum_{m=0}^{\infty}k_{\alpha, \beta}(m,n)t^m (1-t)^n[w_{1}(m)]^{-1}\cdot|a_m|w_{1}(m). \nonumber
\end{equation}
Moreover, we have, for each $t\in (0,1)$,
\begin{eqnarray}
 k_{\alpha, \beta}(m,n)t^m (1-t)^n[w_{1}(m)]^{-1}&=&k_{\alpha, \beta}(m,n)t^m (1-t)^n (m+1)_{\alpha}(m+1)_{\beta}^{-1} \nonumber \\
 &=& \frac{\Gamma(n+m+\beta+1)}{\Gamma(m+\beta+1)\Gamma(n+\beta-\alpha+1)}t^m (1-t)^n. \nonumber
\end{eqnarray}
By Stirling's formula, we know that
\begin{equation}
\frac{\Gamma(n+m+\beta+1)}{\Gamma(m+\beta+1)}=(m+\beta+1)^{n}[1+o(1)],\,\,\, {\text{as}}\,\,\, m\rightarrow \infty. \nonumber 
\end{equation} 
Here and later, $o(1)$ denotes some sequence $\{{\bf a}_i\}_{i=0}^{\infty}$ with $o(1)={\bf a}_i\rightarrow 0$ as $i \rightarrow \infty$, which will be different in different places. 
Hence, for $t\in (0,1)$,
 $$\frac{\Gamma(n+m+\beta+1)}{\Gamma(m+\beta+1)}t^m \rightarrow 0,\,\,\,{\text{as}}\,\,\, m\rightarrow \infty.$$
Then, for each $t\in (0,1)$, there is a constant $M>0$ such that
\begin{equation}
E_n(t)\leq M\frac{(1-t)^n}{\Gamma(n+\beta-\alpha+1)}\sum_{m=0}^{\infty}|a_m|w_{1}(m). \nonumber
\end{equation}
This means that $E_{n}(t)$ is well-defined on $(0,1)$ for each $n\in \mathbb{N}_0$ when $p=1$.

When $p>1$, by H\"older's inequality, we have 
\begin{eqnarray}\label{w-1}
E_n(t)&=& \sum_{m=0}^{\infty}\Big[k_{\alpha, \beta}(m,n)t^m (1-t)^n[w_{1}(m)]^{-\frac{1}{p}}\Big]\Big[a_m [w_{1}(m)]^{\frac{1}{p}}\Big]  \nonumber \\
&\leq & \Big\{\sum_{m=0}^{\infty}\Big[k_{\alpha, \beta}(m,n)t^m (1-t)^n[w_{1}(m)]^{-\frac{1}{p}}\Big]^{q}\Big\}^{\frac{1}{q}}\|a\|_{p,w_1}\nonumber \\
&=& \Big\{\sum_{m=0}^{\infty} \Big[\frac{\Gamma(n+m+\beta+1)}{\Gamma(m+\beta+1)}\Big]^q \frac{\Gamma(m+\beta+1)}{\Gamma(m+1)} t^{qm}\Big\}^{\frac{1}{q}}\frac{(1-t)^{n}}{\Gamma(n+\beta-\alpha+1)}\|a\|_{p,w_1}.
\end{eqnarray}
We note that, using Stirling's formula again, we have for $m\geq 1$,
\begin{eqnarray}\lefteqn{\Big\{\Big[\frac{\Gamma(n+m+\beta+1)}{\Gamma(m+\beta+1)}\Big]^q \frac{\Gamma(m+\beta+1)}{\Gamma(m+1)} t^{qm}\Big\}^{\frac{1}{m}}}\nonumber \\
&&=t^q\Big[\frac{\Gamma(n+m+\beta+1)}{\Gamma(m+\beta+1)}\Big]^{\frac{q}{m}}\Big[\frac{\Gamma(m+\beta+1)}{\Gamma(m+1)}\Big]^{\frac{1}{m}}\nonumber \\
&&= t^q (m+\beta)^{\frac{nq}{m}}[1+o(1)]^{\frac{q}{m}}m^{\frac{\beta}{m}}[1+o(1)]^{\frac{1}{m}}\nonumber \\
&&\rightarrow t^q<1\,\,\, {\text{as}}\,\,\, m\rightarrow \infty.\nonumber
\end{eqnarray}
From the root test for series, we know that, for $t\in(0,1)$, the series 
\begin{equation} \sum_{m=0}^{\infty} \Big[\frac{\Gamma(n+m+\beta+1)}{\Gamma(m+\beta+1)}\Big]^q \frac{\Gamma(m+\beta+1)}{\Gamma(m+1)} t^{qm} \nonumber 
\end{equation}
is convergent. It then follows from (\ref{w-1}) that $E_{n}(t)$ is well-defined on $(0,1)$ for each $n\in \mathbb{N}_0$ when $p>1$.

We shall give an estimate for $E_n^p(t), p\geq 1$. When $p>1$, we have
\begin{eqnarray}
\lefteqn{E_n(t)=\sum_{m=0}^{\infty}[k_{\alpha, \beta}(m,n)t^m (1-t)^n]^{\frac{1}{p}}|a_m|(m+1)_{\alpha}^{-\frac{1}{q}}} \nonumber \\
&& \qquad\quad\quad \times [k_{\alpha, \beta}(m,n)t^m (1-t)^n]^{\frac{1}{q}}(m+1)_{\alpha}^{\frac{1}{q}} \nonumber \\
&& \leq \Big[\sum_{m=0}^{\infty}k_{\alpha, \beta}(m,n)t^m (1-t)^n|a_m|^p (m+1)_{\alpha}^{1-p} \Big]^{\frac{1}{p}}\nonumber \\
&& \qquad\quad\quad \times \Big[\sum_{m=0}^{\infty}k_{\alpha, \beta}(m,n)t^m (1-t)^n(m+1)_{\alpha}\Big]^{\frac{1}{q}}.\nonumber 
\end{eqnarray}
It follows from Lemma \ref{l-2} that 
\begin{eqnarray}\label{ee-1}
\lefteqn{E_n^p(t)\leq (1-t)^{-(\beta+1)(p-1)}(n+\beta-\alpha+1)_{\alpha}^{p-1}}\nonumber \\
&&\times\Big[\sum_{m=0}^{\infty}k_{\alpha, \beta}(m,n)t^m (1-t)^n|a_m|^p(m+1)_{\alpha}^{1-p}\Big]. 
\end{eqnarray}
When $p=1$, the inequality (\ref{ee-1}) obviously holds. 

Hence, for $p\geq 1$, and $t\in (0,1)$, we have
\begin{eqnarray}
\lefteqn{\sum_{n=0}^{\infty}w_2(n)E_n^p(t)=\sum_{n=0}^{\infty}(n+\beta-\alpha+1)_{\alpha}^{-p} (n+1)_{\beta}E_n^p(t)}\nonumber \\
&& \leq (1-t)^{-(\beta+1)(p-1)}\sum_{n=0}^{\infty}(n+\beta-\alpha+1)_{\alpha}^{-1}(n+1)_{\beta} \nonumber \\
&& \qquad\quad\quad \times\Big[\sum_{m=0}^{\infty}k_{\alpha, \beta}(m,n)t^m (1-t)^n|a_m|^p(m+1)_{\alpha}^{1-p}\Big] \nonumber \\
&&= (1-t)^{-(\beta+1)(p-1)}\sum_{m=0}^{\infty}|a_m|^p(m+1)_{\alpha}^{1-p}\Big[\sum_{n=0}^{\infty}k_{\alpha, \beta}(m,n)t^m (1-t)^n (n+1)_{\beta-\alpha}\Big].\nonumber 
\end{eqnarray}
 By using Lemma \ref{l-2}, we get that
\begin{eqnarray}
 \sum_{n=0}^{\infty}w_2(n)E_n^p(t)&\leq&  (1-t)^{-(\beta+1)(p-1)}t^{-\beta-1}\sum_{m=0}^{\infty}|a_m|^p(m+1)_{\alpha}^{1-p}(m+\alpha+1)_{\beta-\alpha}
 \nonumber \\
 &=&  (1-t)^{-(\beta+1)(p-1)}t^{-\beta-1}\sum_{m=0}^{\infty}|a_m|^p(m+1)_{\alpha}^{-p}(m+1)_{\beta}
 \nonumber \\
 &=& (1-t)^{-(\beta+1)(p-1)}t^{-\beta-1}\sum_{m=0}^{\infty}|a_m|^p w_1(m).  \nonumber\end{eqnarray}

We let 
$$A_n:=\sum_{m=0}^{\infty}\int_{(0,1)}k_{\alpha, \beta}(m,n)t^m (1-t)^n|a_m|d\mu(t).$$
Then we have \begin{eqnarray}
A_n&=& \int_{(0,1)}\sum_{m=0}^{\infty}k_{\alpha, \beta}(m,n)t^m (1-t)^n|a_m|d\mu(t) \nonumber \\
&=&  \int_{(0,1)}E_n(t)d\mu(t), \nonumber
\end{eqnarray}
so that 
\begin{eqnarray}
\|H_{\mu}^{\alpha, \beta}(a)\|_{p, w_2}&=& \Big[\sum_{n=0}^{\infty}w_2(n)|H_{\mu}^{\alpha, \beta}(a)(n)|^{p}\Big]^{\frac{1}{p}}\nonumber  \\
&\leq&  \Big[\sum_{n=0}^{\infty}w_2(n)A_n^{p}\Big]^{\frac{1}{p}}=\Big[\sum_{n=0}^{\infty}w_2(n)\Big(\int_{(0,1)}E_n(t)d\mu(t)\Big)^{p}\Big]^{\frac{1}{p}}. \nonumber
\end{eqnarray}
It follows from Minkowski's inequality that
\begin{eqnarray}
\|H_{\mu}^{\alpha, \beta}(a)\|_{p, w_2}&\leq & \int_{(0,1)}\Big(\sum_{n=0}^{\infty}w_2(n)E_n^p(t)\Big)^{\frac{1}{p}}d\mu(t)\nonumber \\
&\leq & \Big[\int_{(0,1)}(1-t)^{-(1-\frac{1}{p})(\beta+1)}t^{-\frac{1}{p}(\beta+1)}d\mu(t)\Big]\Big[\sum_{m=0}^{\infty}|a_m|^p w_1(m)\Big]^{\frac{1}{p}}\nonumber \\
&=& \mathcal{C}_{\mu}(\beta,p)\|a\|_{p, w_1}. \nonumber 
\end{eqnarray}
This proves that $H_{\mu}^{\alpha, \beta}: l_{w_1}^{p} \rightarrow l_{w_2}^{p}$ is bounded when $\mathcal{C}_{\mu}(\beta,p)<\infty$ and $\|H_{\mu}^{\alpha, \beta}\|\leq \mathcal{C}_{\mu}(\beta,p).$ Proposition \ref{i} is proved.

\subsection{Proof of Proposition \ref{ii}}  
Let $\varepsilon\in (0, \frac{1}{p}(\beta+1))$, we define the sequence $a=\{a_m\}_{m=0}^{\infty}$ as
\begin{equation}\label{am}a_m:=(m+1)_{\alpha}(m+1)_{\beta}^{-\frac{1}{p}}(m+\beta+1)^{-(\frac{1}{p}+\varepsilon)}, \,m\in \mathbb{N}_0.\end{equation}
Then 
$$\sum_{m=0}^{\infty}w_1(m)a_m^p=\sum_{m=0}^{\infty}(m+\beta+1)^{-1-p\varepsilon}<\infty.$$

To prove Proposition \ref{ii}, we need the following result. 
\begin{claim}\label{cla}
Let $\rho, \eta\in (0,1/2)$. For any $\varepsilon\in (0, \frac{1}{p}(\beta+1))$, we can find a natural number ${\bf{N}}_0={\bf{N}}_0(\rho, \eta)$ such that 
\begin{equation}
[w_2(n)]^{\frac{1}{p}}A_n\geq (1-\rho)^2[w_1(n)]^{\frac{1}{p}}a_n\int_{\eta}^{1-\eta}(1-t)^{-(1-\frac{1}{p})(\beta+1)+\varepsilon}t^{-\frac{1}{p}(\beta+1)-\varepsilon}d\mu(t),\nonumber \end{equation} 
for all $n\geq {\bf{N}_0}.$  
\end{claim}
\begin{proof}[Proof of Claim \ref{cla}]
We first fix two positive numbers $b$ and $\tau$. We let $$b:=\frac{1}{p}(\beta+1)+\varepsilon,$$
and 
$$\tau:=\tau(\rho)=(1-\rho)^{-\frac{p}{2(\beta+1)}}-1.$$

By Lemma \ref{l-3}, we have
\begin{equation}\label{eq-1}
  a_m \geq (m+1)_{\alpha} (m+1)^{-\frac{1}{p}(\beta+1)}{(m+\beta+1)^{-\varepsilon}}\geq (m+1)_{\alpha} (m+1)^{-\frac{1}{p}(\beta+1)-\varepsilon}, \nonumber
\end{equation}
for $\beta\in (-1,0)$, and 
\begin{equation} \label{eq-2}
  a_m \geq (m+1)_{\alpha} (m+\beta+1)^{-\frac{1}{p}(\beta+1)-\varepsilon},\nonumber
\end{equation}
for $\beta\geq 0$. 
By Lemma \ref{l-4}, we obtain that 
$$a_m\geq \frac{(m+1)_{\alpha}}{\Gamma(b)}\int_{0}^{\infty}e^{-(m+1)x}x^{b-1}dx,$$
for $\beta\in (-1,0)$, and 
$$a_m\geq \frac{(m+1)_{\alpha}}{\Gamma(b)}\int_{0}^{\infty}e^{-(m+\beta+1)x}x^{b-1}dx,$$
for $\beta\geq 0$. 

Thus, for $\beta\in (-1,0)$,  we have
\begin{eqnarray}
E_n(t)&=&\sum_{m=0}^{\infty}k_{\alpha, \beta}(m,n)t^m (1-t)^n |a_m| \nonumber \\
&\geq & \frac{(1-t)^n}{\Gamma(b)}\int_{0}^{\infty} e^{-x} x^{b-1} \sum_{m=0}^{\infty}k_{\alpha, \beta}(m,n)t^m e^{-mx}(m+1)_{\alpha}dx.  \nonumber
\end{eqnarray}
It follows from Lemma \ref{l-2} that
\begin{eqnarray}
E_n(t)&\geq &(n+\beta-\alpha+1)_{\alpha}\frac{(1-t)^n}{\Gamma(b)}\int_{0}^{\infty} e^{-x} x^{b-1} (1-te^{-x})^{-n-\beta-1}dx \nonumber \\
&=&(n+\beta-\alpha+1)_{\alpha} \frac{(1-t)^{-\beta-1}}{\Gamma(b)}\int_{0}^{\infty} e^{-x} x^{b-1} \Big(\frac{1-t}{1-te^{-x}}\Big)^{n+\beta+1}dx. \nonumber 
\end{eqnarray}
Then, from Lemma \ref{l-5} and again Lemma \ref{l-4}, we obtain that
\begin{eqnarray}\label{e-1}
E_n(t)&\geq & (n+\beta-\alpha+1)_{\alpha}\frac{(1-t)^{-\beta-1}}{\Gamma(b)}\int_{0}^{\infty} e^{-x} x^{b-1} e^{-(n+\beta+1)t x {(1-t)}^{-1}}dx  \nonumber \\
&=&(n+\beta-\alpha+1)_{\alpha}(1-t)^{-\beta-1}\Big[\frac{1+(n+\beta)t}{1-t}\Big]^{-b}:=F_{n}(t),
\end{eqnarray}
for $\beta\in (-1,0)$. For $\beta\geq 0$, in the same way, we can obtain that
\begin{eqnarray}\label{e-2}
E_n(t)\geq (n+\beta-\alpha+1)_{\alpha}(1-t)^{-\beta-1}\Big[\frac{\beta+1+nt}{1-t}\Big]^{-b}:=G_{n}(t).
\end{eqnarray}

Hereafter we will assume $n\geq 1$. Then, in order to estimate $F_n(t)$, note that when $t\geq \frac{1}{\tau(n+\beta)}$, it holds that 
\begin{equation}
\frac{1+(n+\beta)t}{t(n+\beta+1)}=1+\frac{1-t}{t(n+\beta+1)}\leq 1+\frac{1}{t(n+\beta)}\leq 1+{\tau}.\nonumber 
\end{equation}
Then it follows from $\varepsilon\in (0, \frac{1}{p}(\beta+1))$ that 
\begin{equation}\label{m-1}\Big[\frac{1+(n+\beta)t}{t(n+\beta+1)}\Big]^{-b}\geq (1+{\tau})^{-\frac{2}{p}(\beta+1)}=1-\rho.\end{equation}
Also, by Stirling's formula, we have  
\begin{equation} (n+1)_{\beta}^{\frac{1}{p}}(n+\beta+1)^{-\frac{\beta}{p}}\rightarrow 1,\,\,\, {\text{as}}\,\,\, n\rightarrow \infty. \nonumber 
\end{equation}

Hence we can take a constant $\bf{N}_1=\bf{N}_1(\rho)\in \mathbb{N}$ such that  
\begin{equation}\label{m-2}
 (n+1)_{\beta}^{\frac{1}{p}}(n+\beta+1)^{-\frac{\beta}{p}}\geq 1-\rho,
\end{equation}
when $n\geq {\bf{N}}_1$. 
Meanwhile, we can write
\begin{eqnarray}\label{m-3}
[w_2(n)]^{\frac{1}{p}}F_n(t)&=&(1-t)^{-(\beta+1)+b}t^{-b}a_n[w_1(n)]^{\frac{1}{p}}\nonumber \\
&& \times \Big[\frac{1+(n+\beta)t}{t(n+\beta+1)}\Big]^{-b} (n+1)_{\beta}^{\frac{1}{p}}(n+\beta+1)^{-\frac{\beta}{p}}.  
\end{eqnarray}
Combining (\ref{m-1}), (\ref{m-2}) and (\ref{m-3}), we obtain that 
\begin{eqnarray}\label{e-3} 
[w_2(n)]^{\frac{1}{p}}F_n(t)\geq (1-\rho)^2(1-t)^{-(\beta+1)+b}t^{-b}a_n[w_1(n)]^{\frac{1}{p}},
\end{eqnarray}
when $n\geq {\bf{N}}_1$ and $t\geq \frac{1}{\tau(n+\beta)}$.

To estimate $G_n(t)$, note that when $t\geq \frac{(\beta+1)}{\tau(n+\beta)}$, we have  
\begin{equation}
\frac{\beta+1+nt}{t(n+\beta+1)}=1+\frac{(1-t)(\beta+1)}{t(n+\beta+1)}\leq 1+\frac{\beta+1}{t(n+\beta)}\leq 1+{\tau}, \nonumber 
\end{equation}
so that 
\begin{equation}\label{n-2} \Big[\frac{\beta+1+nt}{t(n+\beta+1)}\Big]^{-b}\geq (1+{\tau})^{-\frac{2}{p}(\beta+1)}=1-\rho.\end{equation}

We write
\begin{eqnarray}\label{n-3}
[w_2(n)]^{\frac{1}{p}}G_n(t)&=&(1-t)^{-(\beta+1)+b}t^{-b}a_n[w_1(n)]^{\frac{1}{p}}\nonumber \\
&& \times \Big[\frac{\beta+1+nt}{t(n+\beta+1)}\Big]^{-b} (n+1)_{\beta}^{\frac{1}{p}}(n+\beta+1)^{-\frac{\beta}{p}}.
\end{eqnarray}
It follows from (\ref{m-2}), (\ref{n-2}) and (\ref{n-3}) that  
\begin{eqnarray}\label{e-4}
[w_2(n)]^{\frac{1}{p}}G_n(t)\geq (1-\rho)^2(1-t)^{-(\beta+1)+b}t^{-b}a_n[w_1(n)]^{\frac{1}{p}},
\end{eqnarray}
when $n\geq {\bf{N}}_1$ and $t\geq \frac{(\beta+1)}{\tau(n+\beta)}$.

Now, for $\beta\in (-1,0)$, note that when
$$n\geq {\bf{N}}_2={\bf{N}_2}(\rho, \eta):=\lceil{(\tau\eta)}^{-1}\rceil+2,$$ 
it holds that $\frac{1}{\tau(n+\beta)}\leq \eta.$ Here, $\lceil x \rceil$ is the ceiling function.
Hence, we obtain from (\ref{e-1}) and (\ref{e-3}) that
\begin{eqnarray}
[w_2(n)]^{\frac{1}{p}}A_n&=&\int_{(0,1)}[w_2(n)]^{\frac{1}{p}}E_n(t)d\mu(t) \nonumber \\
&\geq& (1-\rho)^2a_n[w_1(n)]^{\frac{1}{p}} \int_{\frac{1}{\tau(n+\beta)}}^{1-\frac{1}{\tau(n+\beta)}}(1-t)^{-(\beta+1)+b}t^{-b}d\mu(t) \nonumber \\
&\geq & (1-\rho)^2 a_n[w_1(n)]^{\frac{1}{p}} \int_{\eta}^{1-\eta}(1-t)^{-(\beta+1)+b}t^{-b}d\mu(t),\nonumber
\end{eqnarray}
when $n\geq\max\{\bf{N}_1, \bf{N}_2\}.$ 

Similarly, for $\beta\geq 0$, when 
$$n\geq {\bf{N}}_3={\bf{N}}_3({\rho, \eta}):=\lceil(\beta+1)(\tau\eta)^{-1}\rceil+1,$$ it holds that
$\frac{\beta+1}{\tau(n+\beta)}\leq\eta,$ so that, from (\ref{e-2}) and (\ref{e-4}), we have 
\begin{eqnarray}
[w_2(n)]^{\frac{1}{p}}A_n&=&\int_{(0,1)}[w_2(n)]^{\frac{1}{p}}E_n(t)d\mu(t) \nonumber \\
&\geq& (1-\rho)^2 a_n[w_1(n)]^{\frac{1}{p}} \int_{\frac{\beta+1}{\tau(n+\beta)}}^{1-\frac{\beta+1}{\tau(n+\beta)}}(1-t)^{-(\beta+1)+b}t^{-b}d\mu(t) \nonumber \\
&\geq & (1-\rho)^2 a_n[w_1(n)]^{\frac{1}{p}} \int_{\eta}^{1-\eta}(1-t)^{-(\beta+1)+b}t^{-b}d\mu(t), \nonumber
\end{eqnarray}
when $n\geq \max\{\bf{N}_1, \bf{N}_3\}.$ 

Then the claim follows if we take ${\bf{N}_0}={\bf{N}_0}({\rho, \eta})=\max\{\mathbf{N}_1, {\mathbf{N}}_2, {\mathbf{N}}_3\}$. The proof of the claim is finished. 
\end{proof}

We proceed to the proof of Proposition \ref{ii}. We still let $\varepsilon\in (0, \frac{1}{p}(\beta+1))$ and take the same $a=\{a_m\}_{m=0}^{\infty}$ as in (\ref{am}). Let $\rho, \delta \in (0,1/2)$, by using Claim \ref{cla}, for any $\varepsilon\in (0, \frac{1}{p}(\beta+1))$, we can take a natural number $\widehat{\bf{N}}_0=\widehat{\bf{N}}_0(\rho, \delta)$ such that 
\begin{equation}\label{z-1}
[w_2(n)]^{\frac{1}{p}}A_n\geq (1-\rho)^2[w_1(n)]^{\frac{1}{p}}a_n\int_{\delta}^{1-\delta}(1-t)^{-(1-\frac{1}{p})(\beta+1)+\varepsilon}t^{-\frac{1}{p}(\beta+1)-\varepsilon}d\mu(t),\end{equation} 
for $n\geq \widehat{\bf{N}}_0.$ By using Lemma \ref{l-6} with respect to $\rho, \widehat{\bf{N}}_0$, we see that there is a constant $\varepsilon_1=\varepsilon_1(\rho, \widehat{\bf{N}}_0(\rho, \delta))\in (0, \frac{1}{p}(\beta+1))$ such that, for $\varepsilon \in (0, \varepsilon_1],$
\begin{equation}\label{z-2}\sum_{n=\widehat{\mathbf{N}}_0+1}^{\infty}w_1(n)a_n^p \geq (1-\rho)\sum_{n=0}^{\infty}w_1(n)a_n^p.\end{equation}
It follows from (\ref{z-1}) and (\ref{z-2}) that, for $\varepsilon \in (0, \varepsilon_1],$
\begin{eqnarray}
\lefteqn{\|H_{\mu}^{\alpha,\beta}(a)\|_{p,w_2}=\Big[\sum_{n=0}^{\infty}w_2(n)A_n^p\Big]^{\frac{1}{p}}\geq \Big[\sum_{n=\widehat{\bf{N}}_0+1}^{\infty}w_2(n)A_n^p\Big]^{\frac{1}{p}}}\nonumber \\
&&\geq  (1-\rho)^{2}\int_{\delta}^{1-\delta}(1-t)^{-(1-\frac{1}{p})(\beta+1)+\varepsilon}t^{-\frac{1}{p}(\beta+1)-\varepsilon}d\mu(t)\Big[\sum_{n=\widehat{\mathbf{N}}_0+1}^{\infty}w_1(n)a_n^p\Big]^{\frac{1}{p}} \nonumber \\
&& \geq (1-\rho)^{2+\frac{1}{p}}\int_{\delta}^{1-\delta}(1-t)^{-(1-\frac{1}{p})(\beta+1)+\varepsilon}t^{-\frac{1}{p}(\beta+1)-\varepsilon}d\mu(t)\|a\|_{p,w_1}. \nonumber 
\end{eqnarray}
Then, for $\varepsilon \in (0, \varepsilon_1],$ \begin{equation}
\|H_{\mu}^{\alpha,\beta}\|\geq (1-\rho)^{2+\frac{1}{p}}\int_{\delta}^{1-\delta}(1-t)^{-(1-\frac{1}{p})(\beta+1)+\varepsilon}t^{-\frac{1}{p}(\beta+1)-\varepsilon}d\mu(t).\nonumber
\end{equation} 

Note that on the interval $[\delta, 1-\delta]$, the functions $$f_{\varepsilon}(t)=(1-t)^{-(1-\frac{1}{p})(\beta+1)+\varepsilon}t^{-\frac{1}{p}(\beta+1)-\varepsilon}, \varepsilon\in (0,\varepsilon_1],$$ are uniformly bounded above and $\mu$ is finite on $[\delta, 1-\delta]$. Then, let $\varepsilon \rightarrow 0^{+}$, by the dominated convergence theorem, we see that
\begin{equation}
\|H_{\mu}^{\alpha,\beta}\|\geq (1-\rho)^{2+\frac{1}{p}}\int_{\delta}^{1-\delta}(1-t)^{-(1-\frac{1}{p})(\beta+1)}t^{-\frac{1}{p}(\beta+1)}d\mu(t).\nonumber
\end{equation} 
Consequently, let $\delta\rightarrow 0^{+}$, we obtain from the monotone convergence theorem that  \begin{equation}
\|H_{\mu}^{\alpha,\beta}\|\geq (1-\rho)^{2+\frac{1}{p}}\int_{0}^{1}(1-t)^{-(1-\frac{1}{p})(\beta+1)}t^{-\frac{1}{p}(\beta+1)}d\mu(t).\nonumber
\end{equation}  Finally, let $\rho \rightarrow 0^{+}$, the boundedness of $\mathcal{C}_{\mu}(\beta, p)$ follows and $\|H_{\mu}^{\alpha,\beta}\|\geq \mathcal{C}_{\mu}(\beta, p)$. This proves Proposition \ref{ii}. 

Now, the proof of Theorem \ref{main} is done.

\section{{\bf Proof of Theorem \ref{main-2}}}
If $\mathcal{C}_{\mu}(\beta, \infty)<\infty$, for $a=\{a_m\}_{m=0}^{\infty}\in l_{\overline{w}_1}^{\infty}$, we have 
\begin{eqnarray}
\lefteqn{\|H_{\mu}^{\alpha, \beta}(a)\|_{\infty, \overline{w}_2}\leq \sup_{n\in \mathbb{N}_0}[\overline{w}_2(n)]A_n} \nonumber \\
&&= \sup_{n\in \mathbb{N}_0}[\overline{w}_2(n)] \int_{(0,1)}\sum_{m=0}^{\infty}k_{\alpha, \beta}(m,n)t^m (1-t)^n |a_m|d\mu(t)\nonumber \\
&&=  \sup_{n\in \mathbb{N}_0}[\overline{w}_2(n)] \int_{(0,1)}\sum_{m=0}^{\infty}k_{\alpha, \beta}(m,n)t^m (1-t)^n(m+1)_{\alpha}\cdot|a_m|[\overline{w}_1(m)]d\mu(t).\nonumber
\end{eqnarray}
Here $A_n$ is the same as in the proof of Theorem \ref{main}. Then, it follows from Lemma \ref{l-2} that
\begin{eqnarray}
\|H_{\mu}^{\alpha, \beta}(a)\|_{\infty, \overline{w}_2} \leq\int_{(0,1)}(1-t)^{-\beta-1}d\mu(t) \|a\|_{\infty, \overline{w}_1}=\mathcal{C}_{\mu}(\beta, \infty) \|a\|_{\infty, \overline{w}_1}.\nonumber 
\end{eqnarray}
This proves that $H_{\mu}^{\alpha, \beta}$ is bounded from $l_{\overline{w}_1}^{\infty}$ into $l_{\overline{w}_2}^{\infty}$, and $\|H_{\mu}^{\alpha, \beta}\|\leq \mathcal{C}_{\mu}(\beta, \infty).$

To prove $\|H_{\mu}^{\alpha, \beta}\|=\mathcal{C}_{\mu}(\beta, \infty).$ We take $a=\{a_m\}_{m=0}^{\infty}$ as
\begin{equation}\label{la-1}a_m=(m+1)_{\alpha},\, m\in \mathbb{N}_0.\end{equation}
Then $\|a\|_{\infty, w_1}=1$, and by Lemma (\ref{l-2}) again, we have
\begin{eqnarray}\label{la-2}
\|H_{\mu}^{\alpha, \beta}(a)\|_{\infty, \overline{w}_2}=\int_{(0,1)}(1-t)^{-\beta-1}d\mu(t) \|a\|_{\infty, \overline{w}_1}=\mathcal{C}_{\mu}(\beta, \infty) \|a\|_{\infty, \overline{w}_1}. 
\end{eqnarray}
This implies that $\|H_{\mu}^{\alpha, \beta}\|=\mathcal{C}_{\mu}(\beta, \infty).$ Also, if $H_{\mu}^{\alpha, \beta}: l_{\overline{w}_1}^{\infty}\rightarrow l_{\overline{w}_2}^{\infty}$ is bounded, take the sequence $a$
as in (\ref{la-1}), and by (\ref{la-2}), we can obtain that $\mathcal{C}_{\mu}(\beta, \infty)<\infty$. This completes the proof of Theorem \ref{main-2}.

\section{{\bf Final remarks}}
In this section, we consider the generalized Hilbert matrix operator acting on the spaces of analytic functions in the unit disk $\mathbb{D}$.   
We denote by $\mathcal{H}(\mathbb{D})$ the class of all analytic functions on $\mathbb{D}$. 
Let $\beta>-1$, $\mu$ be a positive finite Borel measure on $(0,1)$.  For $f=\sum_{m=0}^{\infty}a_mz^m \in  \mathcal{H}(\mathbb{D})$, we formally define 
$$H_{\mu}^{\beta}(f)(z):=\sum_{n=0}^{\infty}\Big[\sum\limits_{m=0}^{\infty}\int_{(0,1)}\binom{n+m+\beta}{m}t^m (1-t)^{n}a_md\mu(t)\Big]z^n,\, z\in \mathbb{D}.$$
When $\beta=0$, the operator $H_{\mu}^{0}$ has been studied in \cite{bcds} and \cite{bd}, where the authors characterize the measure $\mu$ such that the operator $H_{\mu}^{0}$ is bounded on Hardy spaces and Bergman spaces. We will consider the operator $H_{\mu}^{\beta}$ acting on the Dirichlet-type space $\mathcal{D}_{\lambda}$. 
 
For $\lambda\in \mathbb{R}$, the Dirichlet-type space $\mathcal{D}_{\lambda}$ is defined as
\begin{equation*}
\mathcal{D}_{\lambda}=\{f\in \mathcal{H}(\mathbb{D}): \|f\|_{\mathcal{D}_{\lambda}}:=(\sum_{n=0}^{\infty}(n+1)^{1-\lambda} |a_{n}|^2)^{\frac{1}{2}}<\infty\}.\end{equation*}
When $\lambda=0$,  $\mathcal{D}_0$ coincides the classic Dirichlet space $\mathcal{D}$, and when $\lambda=1$, $\mathcal{D}_1$ becomes the Hardy space $H^2$. More generalized Dirichlet-type spaces can be found in \cite{osca}.
 
Note that, for $n\geq 0, s>-1$, there exist two positive constants $C_1, C_2$, which are independent of the number $n$, such that $C_1 (n+1)^s \leq (n+1)_{s}\leq C_2 (n+1)^s$. Then, we see from (1) of Corollary \ref{cor} that
 \begin{theorem}

Let $\beta>-1$ and let $\mu$ be a positive finite Borel measure on $(0,1)$. Let $H_{\mu}^{\beta}$ be as above. Then $H_{\mu}^{\beta}$ is bounded on $\mathcal{D}_{1-\beta}$ if and only if 
$$\mathcal{C}_{\beta}(\mu):=\int_{(0,1)} (1-t)^{-\frac{1}{2}(\beta+1)} t^{-\frac{1}{2}(\beta+1)}d\mu(t)<\infty.$$
Moreover, if $\beta=1$, and 
$$\mathcal{C}_1(\mu)=\int_{(0,1)} (1-t)^{-1} t^{-1}d\mu(t)<\infty,$$ 
then $H_{\mu}^{1}$ is bounded on $\mathcal{D}$ and the norm of $H_{\mu}^{1}$ is $\mathcal{C}_1(\mu),$ and if $\beta=0$,
$$\mathcal{C}_0(\mu)=\int_{(0,1)} (1-t)^{-\frac{1}{2}} t^{-\frac{1}{2}}d\mu(t)<\infty,$$ 
then $H_{\mu}^{0}$ is bounded on $H^2$ and the norm of $H_{\mu}^{0}$ is $\mathcal{C}_0(\mu).$ 
\end{theorem}
We end the paper with the following general problem:
\begin{problem}
Characterize the measure $\mu$ such that $H_{\mu}^{\beta}$ is bounded from one analytic function space $X$ to another one $Y$ in $\mathbb{D}.$
\end{problem}

\section{\bf {Acknowledgments}}
The author would like to thank the referees for careful reading of the paper and invaluable comments and suggestions.


\end{document}